\documentclass{article}
\usepackage[english]{babel}
\usepackage[cp1251]{inputenc}
\usepackage{amsfonts,amssymb,mathrsfs,amscd,amsmath,amsthm}
\usepackage{verbatim}

\usepackage[mathscr]{eucal}
\usepackage[dvips]{graphics}
\usepackage[dvips]{graphicx}
\usepackage{floatflt}
\usepackage[T1]{fontenc}
\usepackage{graphicx}
\usepackage[dvipdfmx]{hyperref}
\usepackage{hyperref}
\usepackage{indentfirst}

\usepackage{url}

\usepackage{mathtools}

\def\ig#1#2#3#4{\begin{figure}[!ht]\begin{center}%
\includegraphics[height=#2\textheight]{#1.eps}\caption{#4}\label{#3}%
\end{center}\end{figure}}

\def\thtext#1{
  \catcode`@=11
  \gdef\@thmcountersep{. #1}
  \catcode`@=12
}

\def\threst{
  \catcode`@=11
  \gdef\@thmcountersep{.}
  \catcode`@=12
}

\theoremstyle{plain}
\newtheorem{thm}{Theorem}[section]
\newtheorem{cor}[thm]{Corollary}

\theoremstyle{definition}
\newtheorem{dfn}[thm]{Definition}

 \catcode`@=11
 \def\.{.\spacefactor\@m}
 \catcode`@=12

\def\R{\mathbb R}

\def\r{\rho}
\def\s{\sigma}

\def\0{\emptyset}
\def\:{\colon}
\def\<{\langle}
\def\>{\rangle}

\def\rom#1{\emph{#1}}
\def\({\rom(}
\def\){\rom)}

\def\ss{\subset}

\def\x{\times}

\def\dis{\operatorname{dis}}

\def\opt{{\operatorname{opt}}}

\def\cH{{\cal H}}

\def\cM{{\cal M}}

\def\cP{{\cal P}}
\def\cR{{\cal R}}

\begin{document}
 \title{Hausdorff Realization of Linear Geodesics of Gromov--Hausdorff Space.}
\author{Alexander~O.~Ivanov, Alexey~A.~Tuzhilin}
\maketitle

\begin{abstract}
We have constructed a realization of rectilinear geodesic (in the sense of~\cite{Memoli2018}), lying in Gromov--Hausdorff space, as a shortest geodesic w.r.t. the Hausdorff distance in an ambient metric space.
\end{abstract}

\section*{Introduction}
\markright{\thesection.~Introduction}
\noindent In~\cite{IvaNikolaevaTuz} it was shown that the space of isometry classes of compact metric spaces, endowed with the Gromov--Hausdorff metric, is a geodesic metric space. The proof consisted of two steps: first, it was shown how an optimal correspondence $R$ between two finite metric spaces $X$ and $Y$ can be endowed with a one-parametric family of metrics generating a shortest geodesic $R_t$, $t\in[0,1]$, connecting $X$ with $Y$, and such that its length equals the Gromov--Hausdorff distance between $X$ and $Y$. At the second step the Gromov Precompactness Criteria was used to prove that for any two compact metric spaces there exists a compact metric space which is their ``midpoint''. A little bit later, in~\cite{Memoli} and independently in~\cite{IvaIliadisTuz}, it was shown that a compact optimal correspondence exists between any two compact metric spaces, and if one defines a one-parametric family of metrics on it in the same manner as it was done in~\cite{IvaNikolaevaTuz}, then again a shortest geodesic with the described properties is obtained. In~\cite{Memoli2018} such geodesics were called \emph{rectilinear\/} (there are some other shortest geodesics thatare called \emph{deviant\/}  in~\cite{Memoli2018}).

In the present paper we show that for each pair $X$, $Y$ of compact metric spaces, and each compact optimal correspondence $R$ between them, the compact $R\x[0,1]$ can be endowed with a metric in such a way that for each $t$ the restriction of the metric to $R\x\{t\}$ coincides with the metric of the compact $R_t$, and that the Hausdorff distance between $R\x\{t\}$ and $R\x\{s\}$ equals to the Gromov--Hausdorff distance between $R_t$ and $R_s$. In other words, we construct a realization of the geodesic $R_t$ as a shortest geodesic w.r.t. the Hausdorff distance defined on the subsets of the space $R\x[0,1]$ endowed with some special metric.

\section{Main Definitions and Preliminary Results}
\markright{\thesection.~Main Definitions and Preliminary Results}
\noindent  Let $X$ be a metric space. By $|xy|$ we denote the distance between points $x,\,y\in X$.

Let $\cP(X)$ be the set of all \textbf{nonempty} subsets of the space $X$. For each $A,\,B\in\cP(X)$, and $x\in X$, we put
\begin{flalign*}
\indent&|xA|=|Ax|=\inf\bigl\{|xa|:a\in A\bigr\},\ |AB|=\inf\bigl\{|ab|:a\in A,\,b\in B\bigr\},&\\
\indent&d_H(A,B)=\max\{\sup_{a\in A}|aB|,\sup_{b\in B}|Ab|\}=\max\bigl\{\sup_{a\in A}\inf_{b\in B}|ab|,\,\sup_{b\in B}\inf_{a\in A}|ba|\bigr\}.
\end{flalign*}

\begin{dfn}
The function $d_H\:\cP(X)\x\cP(X)\to\R$ is called the \emph{Hausdorff distance}.
\end{dfn}

The set of all nonempty closed bounded subsets of the metric space $X$ is denoted by $\cH(X)$.

\begin{thm}[\cite{BurBurIva}]
The function $d_H$ is a metric on $\cH(X)$.
\end{thm}

\begin{thm}[\cite{BurBurIva}]\label{thm:Hausd_dist_comp}
The space $\cH(X)$ is compact iff $X$ is compact.
\end{thm}

Let $X$ and $Y$ be metric spaces. A triple $(X',Y',Z)$ consisting of a metric space $Z$ and two its subsets $X'$ and $Y'$ isometric to $X$ and $Y$, respectively, is called a \emph{realization of the pair $(X,Y)$}. The \emph{Gromov--Hausdorff distance $d_{GH}(X,Y)$ between $X$ and $Y$} is the infimum of the reals $r$ such that there exist realizations $(X',Y',Z)$ of the pair $(X,Y)$ with $d_H(X',Y')\le r$.

\begin{thm}[\cite{BurBurIva}]
The Gromov--Hausdorff distance is a metric on the set of all isometry classes of compact metric spaces.
\end{thm}

\begin{dfn}
The set of all isometry classes of compact metric spaces endowed with the Gromov--Hausdorff metric is called the \emph{Gromov--Hausdorff space\/} and is denoted by $\cM$.
\end{dfn}

We need one more (equivalent) definition of the Gromov--Hausdorff distance. Recall that a \emph{relation\/} between sets $X$ and $Y$ is a subset of the Cartesian product $X\x Y$. The set of all \textbf{nonempty\/} relations between $X$ and $Y$ we denote by $\cP(X,Y)$.

\begin{dfn}
If $X$ and $Y$ are metric spaces, then the \emph{distortion $\dis\s$ of a relation $\s\in\cP(X,Y)$} is the value
$$
\dis\s=\sup\Bigl\{\bigl||xx'|-|yy'|\bigr|: (x,y),\,(x',y')\in\s\Bigr\}.
$$
\end{dfn}

A relation $R\ss X\x Y$ between sets $X$ and $Y$ is called a \emph{correspondence}, if the restrictions to $R$ of the canonical projections $\pi_X\:(x,y)\mapsto x$ and $\pi_Y\:(x,y)\mapsto y$ are surjective. The set of all correspondences between $X$ and $Y$ we denote by $\cR(X,Y)$.

\begin{thm}[\cite{BurBurIva}]\label{th:GH-metri-and-relations}
For any metric spaces $X$ and $Y$ we have
$$
d_{GH}(X,Y)=\frac12\inf\bigl\{\dis R:R\in\cR(X,Y)\bigr\}.
$$
\end{thm}

For topological spaces $X$ and $Y$, we consider $X\x Y$ as the topological space with the standard Cartesian product topology. Then it makes sense to talk about \emph{closed relations and correspondences}. The set of all closed correspondences between $X$ and $Y$ we denote by $\cR^c(X,Y)$.

\begin{cor}[\cite{IvaIliadisTuz}]\label{cor:dis_closed}
For metric spaces $X$ and $Y$ we have
$$
d_{GH}(X,Y)=\frac12\inf\bigl\{\dis R:R\in\cR^c(X,Y)\bigr\}.
$$
\end{cor}

\begin{dfn}
A correspondence $R\in\cR(X,Y)$ is called \emph{optimal\/} if $d_{GH}(X,Y)=\frac12\dis R$. The set of all optimal correspondences between $X$ and $Y$ is denoted by $\cR_\opt(X,Y)$. The subset of $\cR_\opt(X,Y)$ consisting of all closed optimal correspondences is denoted by $\cR_\opt^c(X,Y)$.
\end{dfn}

\begin{thm}[\cite{IvaIliadisTuz}]\label{thm:optimal-correspondence-exists}
For any $X,\,Y\in\cM$ we have $\cR^c_\opt(X,Y)\ne\0$.
\end{thm}

\begin{thm}[\cite{Memoli}, \cite{IvaIliadisTuz}]
For any $X,\,Y\in\cM$ and each $R\in\cR^c_\opt(X,Y)$ the family $R_t$, $t\in[0,1]$, of compact metric spaces such that $R_0=X$, $R_1=Y$, and for $t\in(0,1)$ the space $R_t$ is the set $R$ with the metric
$$
\bigl|(x,y),(x',y')\bigr|_t=(1-t)|xx'|+t\,|yy'|,
$$
is a shortest curve in $\cM$ connecting $X$ and $Y$, and its length is equal to $d_{GH}(X,Y)$.
\end{thm}

In~\cite{Memoli2018} the curve $R_t$ was called a \emph{rectilinear geodesic} corresponding to $R\in\cR^c_\opt(X,Y)$. In the same paper it was noted that the Gromov--Hausdorff space has non rectilinear shortest geodesics.

The main result of the present paper follows from more general theorem describing a special construction of metric on the Cartesian product of a metric space and a segment of a Euclidean line.

\section{A special extension of a metric to Cartesian product}
\markright{\thesection.~A special extension of a metric to Cartesian product}
\noindent  Let $Z$ be an arbitrary set, and $\r_t$, $t\in[a,b]$, a one-parametric family of metrics on $Z$. For convenience, we put $|zz'|_t=\r_t(z,z')$ and $Z_t=Z\x\{t\}$. Fix some $c>0$, and define a distance function on $Z\x[a,b]$ as follows:
\begin{equation}\label{eq:dist}
\bigl|(z_1,t_1)(z_2,t_2)\bigr|=\inf_{z\in Z}\bigl(|z_1z|_{t_1}+|zz_2|_{t_2}\bigr)+c|t_1-t_2|.
\end{equation}
It is clear that this function is non negative and symmetric, i.e., it is a distance function such that its restriction to each section $Z\x\{t\}$ coincides with $\r_t$, and on the section $\{z\}\x[a,b]$ with the Euclidean distance $c|t-s|$, $t,s\in[a,b]$ (notice that the latter one does not depend on the choice of $z$); also, $|Z_tZ_s|=c|t-s|$.

To ensure the triangle inequality, one needs two more conditions.

\begin{thm}\label{thm:metric_constr}
Under the notations introduced above, suppose also that the following conditions hold\/\rom:
\begin{enumerate}
\item\label{thm:metric_constr:1} for any $z,z'\in Z$ the funciton $f(t)=|zz'|_t$, $t\in[a,b]$, is monotonic on $t$\rom;
\item\label{thm:metric_constr:2} for any $(z,t),\,(z',s)\in Z\x[a,b]$ we have
$$
|zz'|_t\le c|t-s|+|zz'|_s+c|t-s|.
$$
\end{enumerate}
Then the distance function defined by Equation~$(\ref{eq:dist})$ satisfies the triangle inequality, i.e., it is a metric\/\rom; moreover, for any $t,s\in[a,b]$ we have $d_H(Z_t,Z_s)=|Z_tZ_s|=c|t-s|$, where $d_H$ is the Hausdorff distance in $\cH\bigl(Z\x[a,b]\bigr)$.
\end{thm}

\begin{proof}
We illustrate the proof in Figure~\ref{fig:proof}.

\ig{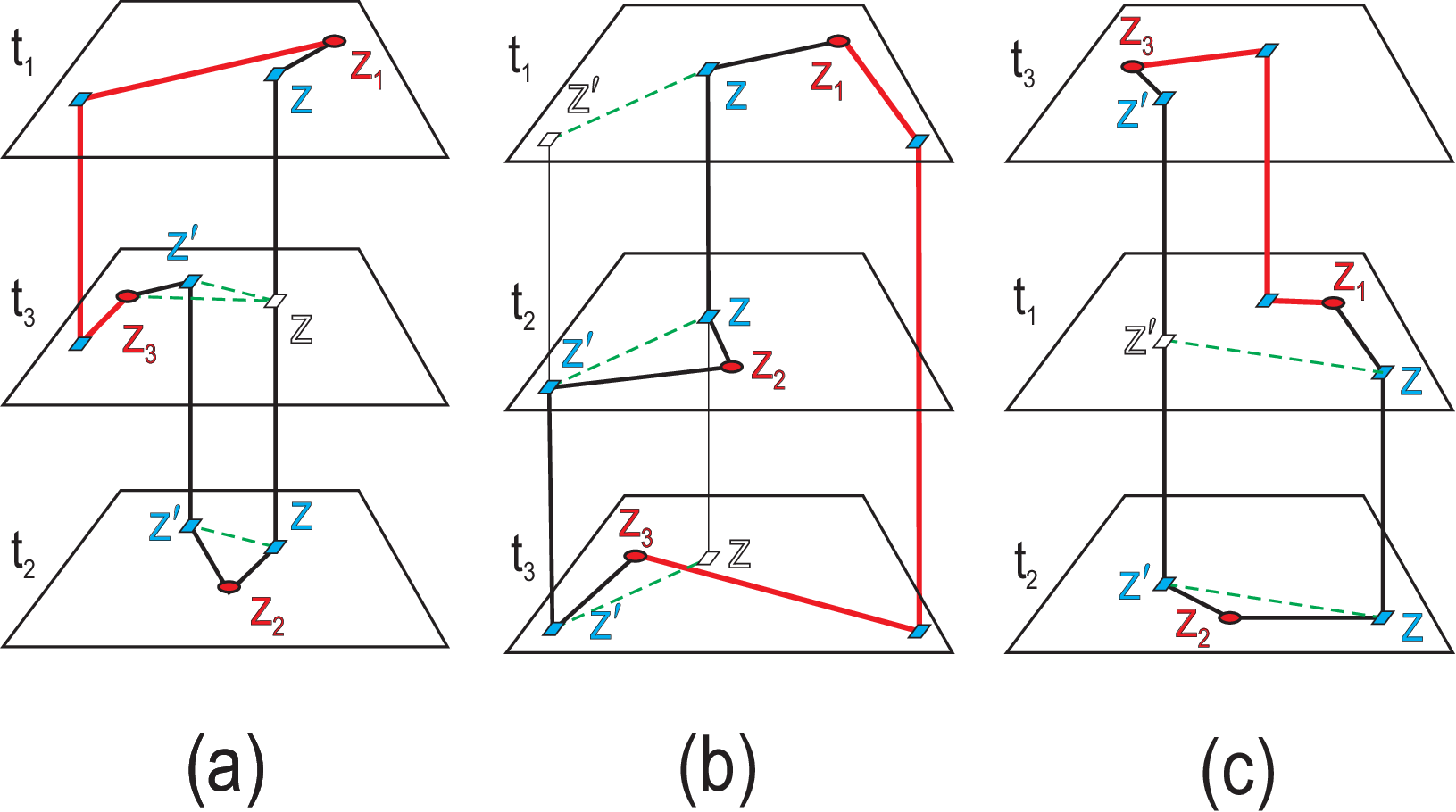}{0.3}{fig:proof}{Triangle inequality verification.}

To verify the triangle inequality, let us choose three arbitrary points $(z_i,t_i)$, $i=1,2,3$ and show that
$$
\bigl|(z_1,t_1)(z_2,t_2)\bigr|+\bigl|(z_2,t_2)(z_3,t_3)\bigr|\ge\bigl|(z_1,t_1)(z_3,t_3)\bigr|.
$$
The proof depends on the ordering of the values $t_i$ (in the figure the $t_i$ correspond to the height of the sections $Z_{t_i}$):  we get three cases up to symmetry and degeneration (when some $t_i$ equal to each other). Now we proceed each of the cases analytically. Notice that the solid polygonal lines connecting $z_i$ and $z_j$ code, in a natural way, the distances between the points $(z_i,t_i)$ and $(z_j,t_j)$; the dashed lines show how one could decrease the value $\bigl|(z_1,t_1)(z_2,t_2)\bigr|+\bigl|(z_2,t_2)(z_3,t_3)\bigr|$ to obtain a one greater or equal to $\bigl|(z_1,t_1)(z_3,t_3)\bigr|$.

The initial steps are the same in all the three cases, and they can be obtained from the following triangle inequality: $|zz_2|_{t_2}+|z_2z'|_{t_2}\ge|zz'|_{t_2}$:
\begin{multline*}
\bigl|(z_1,t_1)(z_2,t_2)\bigr|+\bigl|(z_2,t_2)(z_3,t_3)\bigr|=\\ =
\inf_{z\in Z}\bigl(|z_1z|_{t_1}+|zz_2|_{t_2}\bigr)+c|t_1-t_2|+
\inf_{z'\in Z}\bigl(|z_2z'|_{t_2}+|z'z_3|_{t_3}\bigr)+c|t_3-t_2|\ge\\ \ge
\inf_{z,z'\in Z}\bigl(|z_1z|_{t_1}+|zz'|_{t_2}+|z'z_3|_{t_3}\bigr)+c|t_1-t_2|+c|t_3-t_2|.
\end{multline*}
Then some differences between the Cases appear.

{\bf Case (a)}. We have
\begin{multline*}
\inf_{z,z'\in Z}\bigl(|z_1z|_{t_1}+|zz'|_{t_2}+|z'z_3|_{t_3}\bigr)+c|t_1-t_2|+c|t_3-t_2|=\\ =
\inf_{z,z'\in Z}\bigl(|z_1z|_{t_1}+|zz'|_{t_2}+|z'z_3|_{t_3}\bigr)+c|t_1-t_3|+c|t_3-t_2|+c|t_3-t_2|\ge\\ \ge
\inf_{z,z'\in Z}\bigl(|z_1z|_{t_1}+|zz'|_{t_3}+|z'z_3|_{t_3}\bigr)+c|t_1-t_3|\ge\\ \ge
\inf_{z\in Z}\bigl(|z_1z|_{t_1}+|zz_3|_{t_3}\bigr)+c|t_1-t_3|=\bigl|(z_1,t_1)(z_3,t_3)\bigr|,
\end{multline*}
where the first equality, according to the given order between $t_i$, follows from $c|t_1-t_2|=c|t_1-t_3|+c|t_3-t_2|$; the first inequality follows from condition~(\ref{thm:metric_constr:2}) of Theorem under consideration, that gives $c|t_3-t_2|+|zz'|_{t_2}+c|t_3-t_2|\ge|zz'|_{t_3}$; the last inequality follows from the triangle inequality $|zz'|_{t_3}+|z'z_3|_{t_3}\ge|zz_3|_{t_3}$.

{\bf Case (b)}. In this case we use Condition~(\ref{thm:metric_constr:1}) of Theorem, and thus we conclude that $|zz'|_{t_2}$ is more or equal than either $|zz'|_{t_1}$, or $|zz'|_{t_3}$. Without loss of generality, let us suppose that $|zz'|_{t_2}\ge|zz'|_{t_1}$, then
\begin{multline*}
\inf_{z,z'\in Z}\bigl(|z_1z|_{t_1}+|zz'|_{t_2}+|z'z_3|_{t_3}\bigr)+c|t_1-t_2|+c|t_3-t_2|\ge\\ \ge
\inf_{z,z'\in Z}\bigl(|z_1z|_{t_1}+|zz'|_{t_1}+|z'z_3|_{t_3}\bigr)+c|t_1-t_3|\ge\\ \ge
\inf_{z'\in Z}\bigl(|z_1z'|_{t_1}+|z'z_3|_{t_3}\bigr)+c|t_1-t_3|=\bigl|(z_1,t_1)(z_3,t_3)\bigr|,
\end{multline*}
where the first inequality can be obtained from $|zz'|_{t_2}\ge|zz'|_{t_1}$ and, according to the given order between $t_i$, from $c|t_1-t_3|=c|t_1-t_2|+c|t_3-t_2|$; the second inequality follows from the triangle inequality  $|z_1z|_{t_1}+|zz'|_{t_1}\ge|z_1z'|_{t_1}$.

{\bf Case (c)}. We have
\begin{multline*}
\inf_{z,z'\in Z}\bigl(|z_1z|_{t_1}+|zz'|_{t_2}+|z'z_3|_{t_3}\bigr)+c|t_1-t_2|+c|t_3-t_2|=\\ =
\inf_{z,z'\in Z}\bigl(|z_1z|_{t_1}+|zz'|_{t_2}+|z'z_3|_{t_3}\bigr)+c|t_1-t_2|+c|t_1-t_2|+c|t_3-t_1|\ge\\ \ge
\inf_{z,z'\in Z}\bigl(|z_1z|_{t_1}+|zz'|_{t_1}+|z'z_3|_{t_3}\bigr)+c|t_3-t_1|\ge\\ \ge
\inf_{z'\in Z}\bigl(|z_1z'|_{t_1}+|z'z_3|_{t_3}\bigr)+c|t_1-t_3|=\bigl|(z_1,t_1)(z_3,t_3)\bigr|,
\end{multline*}
where the first equality, according to the given order between $t_i$, follows from the condition $c|t_3-t_2|=c|t_3-t_1|+c|t_1-t_2|$; the first inequality can be obtained from Condition~(\ref{thm:metric_constr:2}) of Theorem, that gives $c|t_1-t_2|+|zz'|_{t_2}+c|t_1-t_2|\ge|zz'|_{t_1}$; the last inequality follows from the triangle inequality $|z_1z|_{t_1}+|zz'|_{t_1}\ge|z_1z'|_{t_1}$.
\end{proof}
As a consequence from Theorem~\ref{thm:metric_constr}, we construct a realization of a rectilinear geodesic in Gromov--Hausdorff space, as a shortest geodesic in the sense of Hausdorff metric.

\section{Realization of Rectilinear Geodesics}
\markright{\thesection.~Realization of Rectilinear Geodesics}
\noindent  Choose arbitrary $X,Y\in\cM$, $R\in\cR_\opt^c(X,Y)$, and construct the corresponding rectilinear geodesic $R_t$, $t\in[0,1]$. For convenience reason, the distance in $R_t$ between points $(x,y)$ and $(x',y')$ is denoted by $\bigl|(x,y),(x',y')\bigr|_t$. Put $c=\frac12\dis R$, and define the distance on $R\x[0,1]$ by Formula~(\ref{eq:dist}).

\begin{cor}
For nonisometric $X$ and $Y$, the distance function on $R\x[0,1]$ defined above is a metric such that $d_H(R_t,R_s)=d_{GH}(X,Y)|t-s|$, thus, $R_t$, being considered as a curve in the space $\cH\bigl(R\x[0,1]\bigr)$, is a shortest curve.
\end{cor}

\begin{proof}
It suffices to verify that the conditions of Theorem~\ref{thm:metric_constr} hold in the case under consideration.

Since $X$ and $Y$ are nonisometric, then $\dis R>0$, thus $c>0$.

Further, for any $(x,y)$ and $(x',y')$ from $R$, the function $f(t)$ from Condition~(\ref{thm:metric_constr:1}) equals $(1-t)|xx'|+t|yy'|$, therefore, it is linear on $t$, thus the Condition~(\ref{thm:metric_constr:1}) holds.

At last, let us check Condition~(\ref{thm:metric_constr:2}). To do that, choose arbitrary $(x,y),(x',y')\in R$, and arbitrary  $t,s\in[0,1]$, then
\begin{multline*}
\bigl|(x,y),(x',y')\bigr|_t-\bigl|(x,y),(x',y')\bigr|_s=\\ =(1-t)|xx'|+t|yy'|-(1-s)|xx'|-s|yy'|=\\ = (t-s)\bigl(|yy'|-|xx'|\bigr)\le|t-s|\dis R=2c|t-s|,
\end{multline*}
that completes the proof.
\end{proof}

\renewcommand\bibname{References}

\end{document}